\documentclass{amsart}

\usepackage{amsmath,amssymb,amsfonts,amsthm} 
\usepackage{MnSymbol} 
\usepackage{mathtools} 
\usepackage{enumitem} 
\usepackage{dsfont} 
\usepackage[cal=euler]{mathalpha} 
\usepackage[protrusion=true,expansion=false]{microtype} 
\usepackage[utf8]{inputenc} 
\usepackage[british]{babel} 
\usepackage[T1]{fontenc} 
\usepackage{mathrsfs} 
\usepackage[pagebackref=true]{hyperref} 
\usepackage{tikz} 
\usetikzlibrary{3d} 

\renewcommand*{\backref}[1]{}
\renewcommand*{\backrefalt}[4]{{\scriptsize(%
    \ifcase #1 Not cited.%
          \or Cited on p.~#2.%
          \else Cited on pp. #2.%
    \fi%
    )}}
   
\definecolor{dblue}{rgb}{0,0,0.70} 
\definecolor{fdblue}{rgb}{0,0,0.70} 
\definecolor{ffgreen}{rgb}{0.1,0.55,0.1} 
\definecolor{fdred}{rgb}{0.6,0,0.08} 

\hypersetup{colorlinks=true,allcolors=dblue}

\setlist[enumerate,1]{label={\textup{\arabic*.}}} 

\newtheoremstyle{boldrk}
  {}{}
  {}{}
  {\bfseries}{.}
  {5pt plus 1pt minus 1pt}{}

\theoremstyle{plain}

\newtheorem{thm}{Theorem}[section]
\newtheorem{claim}{Claim}[thm]

\newtheorem{fact}[thm]{Fact}

\newtheorem{prop}[thm]{Proposition}

\newtheorem*{mainthm}{Main Theorem}

\theoremstyle{boldrk}

\newtheorem{egnum}[thm]{Example}
\newtheorem{qn}[thm]{Question}

\newtheorem*{eg}{Example}

\newtheorem*{rk}{Remark}

\theoremstyle{definition}

\newtheorem{defn}[thm]{Definition}

\newtheorem*{defnast}{Definition}

\renewcommand{\leq}{\leqslant}

\renewcommand{\epsilon}{\varepsilon}
\renewcommand{\subset}{\subseteq}
\renewcommand{\supset}{\supseteq}

\newcommand{\bbP}{\ensuremath{\mathbb{P}}}
\newcommand{\bbQ}{\ensuremath{\mathbb{Q}}}

\newcommand{\dsP}{\mathbb{P}}

\newcommand{\ddbbP}{\ensuremath{\dot{\bbP}}}
\newcommand{\ddbbQ}{\ensuremath{\dot{\bbQ}}}

\newcommand{\calI}{\ensuremath{\mathcal{I}}}

\newcommand{\sF}{\ensuremath{\mathcal{F}}}
\newcommand{\sG}{\ensuremath{\mathcal{G}}}

\newcommand{\dsF}{\ensuremath{\dot{\sF}}}
\newcommand{\dsG}{\ensuremath{\dot{\sG}}}

\newcommand{\dsI}{\ensuremath{\dot{\sI}}}

\newcommand{\frakF}{\ensuremath{\mathfrak{F}}}
\newcommand{\frakG}{\ensuremath{\mathfrak{G}}}

\newcommand{\power}{\ensuremath{\mathscr{P}}}
\newcommand{\sI}{\ensuremath{\mathscr{T}}} 
\newcommand{\sS}{\ensuremath{\mathscr{S}}}

\newcommand{\dsS}{\ensuremath{\dot{\sS}}}

\newcommand{\1}{\ensuremath{\mathds{1}}}

\newcommand{\AC}{\ensuremath{\mathsf{AC}}}
\newcommand{\HS}{\ensuremath{\mathsf{HS}}}
\newcommand{\KWP}{\ensuremath{\mathsf{KWP}}}
\newcommand{\ZF}{\ensuremath{\mathsf{ZF}}}
\newcommand{\ZFC}{\ensuremath{\mathsf{ZFC}}}

\DeclareMathOperator{\Add}{Add}
\DeclareMathOperator{\Aut}{Aut}

\DeclareMathOperator{\Col}{Col}
\DeclareMathOperator{\dom}{dom}
\DeclareMathOperator{\fix}{fix}
\DeclareMathOperator{\id}{id}
\DeclareMathOperator{\Ord}{Ord}
\DeclareMathOperator{\SC}{SC}
\DeclareMathOperator{\supp}{supp}
\DeclareMathOperator{\sym}{sym}

\newcommand{\dda}{\ensuremath{\dot{a}}}
\newcommand{\ddb}{\ensuremath{\dot{b}}}
\newcommand{\ddf}{\ensuremath{\dot{f}}}

\newcommand{\ddq}{\ensuremath{\dot{q}}}

\newcommand{\ddt}{\ensuremath{\dot{t}}}
\newcommand{\ddx}{\ensuremath{\dot{x}}}
\newcommand{\ddy}{\ensuremath{\dot{y}}}

\newcommand{\ddA}{\ensuremath{\dot{A}}}
\newcommand{\ddB}{\ensuremath{\dot{B}}}
\newcommand{\ddG}{\ensuremath{\dot{G}}}
\newcommand{\ddH}{\ensuremath{\dot{H}}}
\newcommand{\ddX}{\ensuremath{\dot{X}}}
\newcommand{\ddY}{\ensuremath{\dot{Y}}}

\newcommand{\abs}[1]{\ensuremath{\mathchoice
	{\left\lvert#1\right\rvert}
	{\lvert#1\rvert}
	{\lvert#1\rvert}
	{\lvert#1\rvert}
}} 
\newcommand{\tup}[1]{\ensuremath{\langle#1\rangle}}

\renewcommand{\colon}{\ensuremath{\nobreak\mskip2mu\mathpunct{}\nonscript
  \mkern-\thinmuskip{:}\mskip6muplus1mu\relax}}

\newcommand{\res}{\ensuremath{\nobreak\mskip2mu\mathpunct{}\nonscript
  \mkern-\thinmuskip{\upharpoonright}\mskip4muplus1mu\relax}}

\newcommand*{\defeq}{\mathrel{\vcenter{\baselineskip0.5ex \lineskiplimit0pt
                     \hbox{\scriptsize.}\hbox{\scriptsize.}}}%
                     =}

\newcommand{\comp}{\ensuremath{\mathrel{\parallel}}} 
\newcommand{\forces}{\ensuremath{\mathrel{\Vdash}}} 
\newcommand{\mint}{\ensuremath{\textstyle\int}} 
\newcommand{\sidecheck}[1]{\ensuremath{#1\,\check{}\,}} 

\newcommand{\concat}{\ensuremath{\mathbin{\raisebox{.9ex}{\scalebox{.7}{\(\frown\)}}}}} 
\newcommand{\lomega}{\ensuremath{{{<}\omega}}} 
\newcommand{\SetSymbol}{\ensuremath{\;\middle|\;}} 
\newcommand{\vphi}{\ensuremath{\varphi}} 

\title[Upwards Homogeneity]{Upwards Homogeneity in\linebreak Iterated Symmetric Extensions}
\date{4th October 2025}
\thanks{The first and third authors' work was financially supported by EPSRC via the Mathematical Sciences Doctoral Training Partnership [EP/W523860/1]. The second author was supported by a UKRI Future Leaders Fellowship [MR/T021705/2]. For the purpose of open access, the authors have applied a Creative Commons Attribution (CC-BY) licence to any Author Accepted Manuscript version arising from this submission. No data are associated with this article.} 

\author[C. Ryan-Smith]{Calliope Ryan-Smith}
\email{cr638@cantab.ac.uk}
\urladdr{https://academic.calliope.mx}

\author[J. Schilhan]{Jonathan Schilhan}
\email{j.Schilhan@leeds.ac.uk}
\urladdr{http://www.logic.univie.ac.at/~schilhanj}

\author[Y. Wei]{Yujun Wei}
\email{mmssh@leeds.ac.uk}

\address{School of Mathematics, University of Leeds, LS2 9JT, UK}

\begin{document}

\begin{abstract}
It is sometimes desirable in choiceless constructions of set theory that one iteratively extends some ground model without adding new sets of ordinals after the first extension. Pushing this further, one may wish to have models \(V\subseteq M\subseteq N\) of \(\ZF\) such that \(N\) contains no subsets of \(V\) that do not already appear in \(M\). We isolate, in the case that \(M\) and \(N\) are symmetric extensions (particular inner models of a generic extension of \(V\)), the exact conditions that cause this behaviour and show how it can broadly be applied to many known constructions. We call this behaviour upwards homogeneity.
\end{abstract}

\maketitle

\section{Introduction}\label{sec:Introduction}

Cantor's \emph{well-ordering principle} is a foundational part of modern formalised mathematics that asserts that every set can be well-ordered. Using some translation, we may equally describe this idea as being that for every set we may construct a bijection between that set and a set of ordinals. The least order type of such a set is called the cardinality of that set and has innumerable applications across mathematics. However, this principle cannot be deduced from the other axioms of Zermelo\nobreakdash--Fraenkel set theory and thus one can create models of \(\ZF\) that do not satisfy this. One method to do so is \emph{symmetric extensions}, a technique that builds on forcing and constructs intermediate models \(V\subseteq M\subseteq V[G]\) for \(V\)-generic filters \(G\). This intermediate model is guaranteed to satisfy \(\ZF\) but may not satisfy the well-ordering principle and, hence, there may be sets that are not in bijection with a set of ordinals.

One generalisation of this view of the well-ordering principle would be to ask that all sets are in bijection with sets \emph{of sets} of ordinals. This also cannot be proved from \(\ZF\) alone, so one could go one step further and ask that all sets are in bijection with sets \emph{of sets of sets} of ordinals, and so on. To this end, we say that an \emph{\(\alpha\)-set} is an element of \(\power^\alpha(\Ord)\), the \(\alpha\)\textup{th} iteration of the power set operator on the class of ordinals taking unions at limit stages. The \emph{Kinna\nobreakdash--Wagner Principle} for \(\alpha\), written \(\KWP_\alpha\), is the assertion that every set is in bijection with a set of \(\alpha\)-sets. In particular, \(\KWP_0\) states that every set is in bijection with a set of ordinals and so is equivalent to the well-ordering principle. One may then hope that \(\ZF\) at least asserts that there is some ordinal \(\alpha\) such that \(\KWP_\alpha\) holds, but this is not the case. An early stage of showing this failure was Monro's iteration (explored in Section~\ref{sec:Examples;subsec:monros-iteration} and originally published in \cite{Monro1973}), a method by which a strictly increasing chain \(\tup{M_n\mid n<\omega}\) of models of \(\ZF\) are constructed such that for all \(k>n\), \(M_k\) and \(M_n\) have the exact same class of \(n\)-sets. In particular, since the chain is strictly increasing, \(M_k\vDash\lnot\KWP_n\). This stands in stark contrast to Balcar and Vop\v{e}nka's theorem that if \(V\) and \(W\) are transitive models of \(\ZF\), \(V\vDash\AC\) or \(W\vDash\AC\), and \(\power(\Ord)^V=\power(\Ord)^W\) then \(V=W\).\footnote{This result can be further generalised to ``if \(V\) and \(W\) are transitive models of \(\ZF\), \(V\vDash\KWP_\alpha\) or \(W\vDash\KWP_\alpha\), and \(\power^{\alpha+1}(\Ord)^V=\power^{\alpha+1}(\Ord)^W\) then \(V=W\)''.} Monro's construction was later extended by Shani in \cite{Shani2018}. Taking the failure of Kinna--Wagner further, the Bristol Model, exhibited in \cite{Karagila2018}, is a model of \(\ZF\) such that \(\ZF\vDash\lnot\KWP_\alpha\) for all ordinals \(\alpha\).

In each case the constructions make use of a certain homogeneity in their forcing.\footnote{Our explanation here shall be somewhat imprecise. All terms and notation used are explained in Section~\ref{sec:Preliminaries}.} For a notion of forcing \(\bbP\) and an automorphism \(\pi\) of \(\bbP\), one can extend \(\pi\) to act on all \(\bbP\)-names and, as one might expect, if \(p\forces\vphi(\ddx)\) for some \(p\in\bbP\) and \(\bbP\)-name \(\ddx\), then \(\pi p\forces\vphi(\pi\ddx)\). Suppose that for all \(q,q'\leq p\) there is \(\pi\) such that \(\pi q=q'\) and \(\pi\ddx=\ddx\). Then if \(q\forces\vphi(\ddx)\) we must have that for all \(q'\leq p\), \(q'\forces\vphi(\ddx)\) as well. Hence we must have that \(p\forces\vphi(\ddx)\) or \(p\forces\lnot\vphi(\ddx)\). By noting that ground-model sets are not moved by any \(\pi\), we have that for all ground-model sets \(y\), either \(p\forces y\in\ddx\) or \(p\forces y\notin\ddx\). Therefore, if \(p\forces\ddx\subseteq V\) then we must have \(p\forces\ddx\in V\). A similar process can be carried out in an iteration of symmetric extensions \(V\subseteq M\subseteq N\) to ensure that if \(\tup{p,\ddq}\forces\ddx\subseteq V\) then \(p\forces\ddx\in M\). That is, \(\power(V)^M=\power(V)^N\). We have isolated a homogeneity condition on iterations of symmetric extensions that precisely characterises when the second step does not add subsets of the ground model.

\begin{defnast}
\(\tup{\bbP_0,\sG_0,\sF_0}\ast\tup{\ddbbP_1,\dsG_1,\dsF_1}^\bullet\) is \emph{upwards homogenous} if for all \(\bar\Gamma\in\sF_0\ast\dsF_1\) there is a dense (equivalently, predense) collection of \(\tup{p^\circ,\ddq^\circ}\in\bbP_0\ast\ddbbP_1\) such that for all \(\tup{p,\ddq},\tup{p,\ddq'}\leq\tup{p^\circ,\ddq^\circ}\) there is \(\bar\pi\in\bar\Gamma\) such that \(\bar\pi\tup{p,\ddq}\comp\tup{p,\ddq'}\).
\end{defnast}

\begin{mainthm}
\(\sS=\tup{\bbP_0,\sG_0,\sF_0}\ast\tup{\ddbbP_1,\dsG_1,\dsF_1}^\bullet\) is upwards homogeneous if and only if for all symmetric extensions \(V\subseteq M\subseteq N\) by \(\sS\), \(\power(V)^M=\power(V)^N\).
\end{mainthm}

This homogeneity is exploited in several choiceless constructions, some of which we shall exhibit here: Monro's model of \(\lnot\KWP_n\) for all \(n<\omega\); Karagila\nobreakdash--Schlicht's `forgetting how to count' construction; Morris's construction of countable unions of countable sets, the power sets of which have arbitrarily large Lindenbaum number; and the Bristol model.

\begin{rk}
It is in the first exhibition of the Bristol model, \cite[Definition~3.20]{Karagila2018}, that the phrase ``upwards homogeneity'' is used, albeit with a different definition. The definition presented here is a modification of the original that allows it to work in all settings.
\end{rk}

\subsection{Structure of the paper}\label{sec:Introduction;subsec:Structure}

We establish preliminaries in Section~\ref{sec:Preliminaries} to fix notation and to introduce the essential mechanics of symmetric extensions and their iterations. In Section~\ref{sec:Upwards-Homogeneity} we define upwards homogeneity and prove our \hyperref[thm:upwards-homogeneous]{main theorem}. In Section~\ref{sec:Examples} we exhibit some examples that make use of this technique, illustrating the idea behind the definition. We end with some open questions in Section~\ref{sec:Future}.

\section{Preliminaries}\label{sec:Preliminaries}

By a \emph{forcing} we mean a partially ordered set \(\tup{\bbP,\leq_{\bbP}}\) with a maximum element \(\1_{\bbP}\). When there is no room for confusion, the subscripts may be omitted from \(\1_{\bbP}\) and \(\leq_{\bbP}\). Elements of \(\bbP\) are called \emph{conditions} and the relation \(q\leq_{\bbP}p\) is referred to as \(q\) \emph{extending} \(p\) or \(q\) being \emph{stronger} than \(p\). Furthermore, we follow the Goldstern alphabet convention: \(p\) is never a stronger condition than \(q\), etc. We say that \(p,p'\in\bbP\) are \emph{compatible}, denoted \(p\comp p'\), if they have a common extension in \(\bbP\).

We recursively define a \(\bbP\)-name as any set of tuples \(\tup{p,\ddx}\), where \(p\in\bbP\) and \(\ddx\) is a \(\bbP\)-name. We say that a \(\bbP\)-name \(\ddy\) \emph{appears in} \(\ddx\) if there is \(p\in\bbP\) such that \(\tup{p,\ddy}\in\ddx\).

For a family \(X=\{ \dot{x}_i\mid i\in I\}\) of \(\dsP\)-names the \emph{bullet name} for \(X\), denoted \(X^\bullet\), is
\begin{equation*}
\{\dot{x}_i\mid i\in I\}^\bullet=\{\tup{\1_{\dsP},\dot{x}_i}\mid i\in I\}.
\end{equation*}
The bullet name of a set provides a canonical name for the set of realisations of \(X\). This extends to ordered pairs, sequences, functions, etc. We recursively define canonical names for ground-model sets as \(\check{x}=\{\check{y}\mid y\in x\}^\bullet\) so that \(\check{x}\), the \emph{check name} for \(x\), is always evaluated to \(x\) in the forcing extension. In the case that the symbols representing a set are too long, such as \(f(x)\), we may put the check to the side instead, such as \(\sidecheck{f(x)}\). A formula in the language of the forcing \(\dsP\) is a formula in the language of set theory but with every free variable replaced by a \(\dsP\)-name.

Throughout this paper the dot notation \(\ddx\), \(\ddy\), etc.\ is reserved for names associated with some notion of forcing. If \(\dot{a}\) is a name, then the undotted \(a\) implicitly denotes the set interpreted by \(\dot{a}\) in the corresponding forcing extension, and vice versa.

\subsection{Symmetric extensions}

If \(V\) is a model of \(\AC\), then any forcing extension of \(V\) will also model \(\AC\), so additional ideas are required to obtain models in which \(\AC\) fails. Symmetric extensions provide a way to do this.

For a notion of forcing \(\bbP\), let \(\Aut(\bbP)\) be the group of automorphisms of \(\bbP\). For \(\pi\in\Aut(\bbP)\) and a \(\bbP\)-name \(\ddx\), we inductively define the action of \(\pi\) on \(\ddx\) by
\begin{equation*}
\pi\ddx=\{\tup{\pi p,\pi\ddy}\mid\tup{p,\ddy}\in\ddx\}.
\end{equation*}
The \emph{Symmetry Lemma}, recited below, is a consequence of this construction. A proof can be found in \cite[Lemma~5.13]{Jech1973}.
\begin{thm}[Symmetry Lemma]
Let \(\pi\in \Aut(\dsP)\). Then for any \(\dsP\)-name \(\dot{x}\) and formula \(\vphi\), \(p\forces \vphi(\dot{x})\) if and only if \(\pi p\forces\vphi(\pi\dot{x})\).\qed
\end{thm}
Let \(\sG\leq\Aut(\bbP)\). A \emph{filter of subgroups} of \(\sG\) is a set \(\sF\) of subgroups of \(\sG\) such that:
\begin{enumerate}
\item \(\sG\in\sF\);
\item if \(\Gamma\in\sF\) and \(\Gamma\leq\Delta\leq\sG\) then \(\Delta\in\sF\); and
\item if \(\Gamma,\Delta\in\sF\) then \(\Gamma\cap\Delta\in\sF\).
\end{enumerate}
\(\sF\) is \emph{normal} if for all \(\Gamma\in\sF\) and \(\pi\in\sG\), \(\pi\Gamma\pi^{-1}\in\sF\). Finally, a \emph{symmetric system} is a triple \(\sS=\tup{\bbP,\sG,\sF}\) such that \(\bbP\) is a notion of forcing, \(\sG\leq\Aut(\bbP)\), and \(\sF\) is a normal filter of subgroups of \(\sG\). Occasionally, to simplify the exposition, we may only require that \(\sG\) is acting on \(\bbP\) rather than literally consisting of maps \(\pi\colon\bbP\to\bbP\). All notions below are then derived in an analogous way.

Given a \(\bbP\)-name \(\ddx\) and a group \(\sG\leq\Aut(\bbP)\), we denote by \(\sym_{\sG}(\ddx)\) the set \(\{\pi\in\sG\mid\pi\ddx=\ddx\}\). Given a filter \(\sF\) of subgroups of \(\sG\), we say that \(\ddx\) is \emph{\(\sF\)-symmetric} (or simply \emph{symmetric}) if \(\sym_{\sG}(\ddx)\in\sF\). The class of all \emph{hereditarily symmetric} names\footnote{That is, \(\ddx\) is \(\sF\)-symmetric and, for all \(\ddy\) appearing in \(\ddx\), \(\ddy\) is hereditarily \(\sF\)-symmetric.} is denoted \(\HS_{\sF}\). If \(\sS=\tup{\bbP,\sG,\sF}\) is a symmetric system then we shall say that \(\ddx\) is an \emph{\(\sS\)-name} to mean that \(\ddx\) is a hereditarily \(\sF\)-symmetric \(\bbP\)-name.

If \(G\) is a \(V\)-generic filter for \(\bbP\), then \(\HS_{\sF}^G\defeq\{\dot{x}^G\mid \dot{x}\in \HS_{\sF}\}\) is a \emph{symmetric extension} over \(V\) given by \(\tup{\bbP,\sG,\sF}\), and is a submodel of \(V[G]\) satisfying \(\ZF\) (see \cite[Theorem~5.14]{Jech1973}). We also say that \(\HS_{\sF}^G\) is a symmetric extension \emph{according to \(\sS\)} or \emph{by \(\sS\)} in this case.

For \(p\in\bbP\) and \(\vphi\) a formula in the language of the forcing, we write \(p\forces_{\sS}\vphi\) to mean that \(p\) forces \(\vphi\) when quantifiers are relativised to the class \(\HS_{\sF}\) (this is often written \(p\forces^\HS\vphi\) in the literature).\footnote{Note that for all \(\Delta_0\) sentences \(\vphi\), such as \(\ddx\in\ddy\), \(p\forces\ddx\in\ddy\) if and only if \(p\forces_\sS\ddx\in\ddy\).} If \(p\forces_\sS\vphi\) or \(p\forces_\sS\lnot\vphi\) then we say that \(p\) has \emph{decided} \(\vphi\). Similar language may be used for more specific applications. For instance when \(\ddf\) is a name for a function with range included in \(V\), we might say that \(p\) decides \(\ddf(\ddx)\) if there is \(y\) such that \(p\forces_\sS\ddf(\ddx)=\check{y}\).

We will often omit subscripts when clear from context, so \(\HS\) means \(\HS_{\sF}\), \(\sym(\ddx)\) means \(\sym_{\sG}(\ddx)\), and so forth.

\begin{egnum}\label{eg:cohens-first-model}
An illustrative example of symmetric extensions is \emph{Cohen's first model}. Let \(\bbP=\Add(\omega,\omega)\), that is the forcing with conditions that are finite partial functions \(p\colon\omega\times\omega\to2\), ordered by reverse inclusion. For \(n<\omega\) let
\begin{equation*}
\dot{a}_n=\{\langle p,\check{k}\rangle\mid p(n,k)=1\}\quad\text{and}\quad\dot{A}=\{\dot{a}_n\mid n<\omega\}^\bullet.
\end{equation*}
The automorphism group \(\sG\) is \(S_{\lomega}\), the finitary permutations of \(\omega\). That is, those bijections \(\pi\colon\omega\to\omega\) such that \(\pi(n)=n\) for all but finitely many \(n\). We define the \(\sG\)-action on \(\bbP\) as
\begin{equation*}
\pi p(\pi n,m)=p(n,m),
\end{equation*}
in which case \(\pi\dot{a}_n=\dot{a}_{\pi n}\) and \(\pi \dot{A}=\dot{A}\).

We define the filter \(\sF\) as being generated by subgroups of the form
\begin{equation*}
\fix(E)=\{\pi\in\sG\mid\pi\res E=\id\}
\end{equation*}
for \(E\in [\omega]^\lomega\). \(\sF\) is normal since \(\pi\fix(E)\pi^{-1}=\fix(\pi``E)\), and so \(\sS=\tup{\bbP,\sG,\sF}\) is a symmetric system. As \(\fix(\{n\})\leq\sym(\dda_n)\) and \(\sG=\sym(\ddA)\) we see that \(\ddA\) and the \(\dda_n\) are all \(\sS\)-names.

\begin{claim}
Whenever \(M\) is a symmetric extension according to \(\sS\), \(A\) is infinite Dedekind-finite. That is, there is no injection \(\omega\to A\) in \(M\).
\end{claim}

\begin{proof}[Proof of Claim]
Firstly, for \(n\neq m<\omega\) the set \(\{p\in\bbP\mid(\exists k)\,p(n,k)\neq p(m,k)\}\) is dense and so for all \(n\neq m\), \(\1\forces_{\sS}\dda_n\neq\dda_m\). Hence \(A\) is infinite.

To see that \(A\) is Dedekind-finite, let \(\ddf\) be an \(\sS\)-name, \(p\in\bbP\) be such that \({p\forces_{\sS}``\ddf\text{ is a function }\check{\omega}\to\ddA"}\), and \(E\in[\omega]^\lomega\) be such that \(\fix(E)\leq\sym(\ddf)\). Let \(q_0\leq p\) and extend \(q_0\) to \(q\) such that for some \(m\notin E\) and \(n<\omega\), \(q\forces_{\sS}\ddf(\check{n})=\dda_m\). Since \(\dom(q)\) is finite, there is \(m'\notin E\) such that for all \(k<\omega\), \(\tup{m',k}\notin\dom(q)\). Setting \(\pi\) to be the transposition \((\begin{matrix}m&m'\end{matrix})\in\fix(E)\) we have that \(\pi q\comp q\) and \(\pi q\forces_{\sS}\pi\ddf(\pi\check{n})=\pi\dda_m\), so \(\pi q\forces_{\sS}\ddf(\check{n})=\dda_{m'}\). Since \(\pi q\comp q\), \(q\) cannot have forced that \(\ddf\) is an injection. Since \(q_0\) was arbitrary below \(p\) this behaviour is dense and \(p\forces_{\sS}``\ddf\text{ is not an injection}"\).
\end{proof}
\end{egnum}

\begin{defn}[Cohen forcing]
We shall make frequent use of the following generalisation of \(\Add(\omega,\omega)\): Given sets \(X\) and \(Y\) we denote by \(\Add(X,Y)\) the notion of forcing of partial functions \(p\colon Y\times X\to2\) such that \(\dom(p)\) is well-orderable and \(\abs{\dom(p)}<\abs{X}\). The ordering is given by \(q\leq p\) if \(q\supseteq p\). A generic filter for \(\Add(X,Y)\) will induce a new function \(c\colon Y\times X\to2\) and so can be thought of as introducing a \(Y\)-indexed family of new subsets of \(X\).
\end{defn}

\subsubsection{Iterating symmetric extensions}

The idea of iterating a symmetric extension, much like iterating forcing extensions, is quite simple at its core. One has a symmetric system \(\sS_0=\tup{\bbP_0,\sG_0,\sF_0}\) and considers the symmetric extension \(M=\HS^{G_0}_{\sF_0}\supseteq V\) given by some \(V\)-generic filter \(G_0\) for \(\sS_0\). Then, in \(M\), one can construct a further symmetric system \(\sS_1=\tup{\bbP_1,\sG_1,\sF_1}\) and can take, over \(M\), a symmetric extension by \(\sS_1\) to obtain \(N=\HS^{G_1}_{\sF_1}\supseteq M\) given by some \(M\)-generic filter \(G_1\) for \(\sS_1\). Analogously to iterating notions of forcing we can describe this process as a single symmetric extension in the ground model. While \cite{Karagila2016} has a thorough treatment of the topic, we shall only go over what is required for accessing our results. In particular, proofs that the constructions are well-defined are omitted in some cases.

\begin{defn}[Two-step symmetric extension]
Let \(\sS_0=\tup{\bbP_0,\sG_0,\sF_0}\) be a symmetric system, and let \(\dsS_1=\tup{\ddbbP_1,\dsG_1,\dsF_1}^\bullet\) be an \(\sS_0\)-name such that \(\1_{\bbP_0}\forces_{\sS_0}``\dsS_1\) is a symmetric system'' and \(\sym(\dsS_1)=\sG_0\). We define the iteration \(\sS=\sS_0\ast\dsS_1\) to be the symmetric system \(\tup{\bbP,\sG,\sF}\) thusly:

\(\bbP\) is defined similarly to the usual iteration of notions of forcing. Conditions in \(\bbP\) are pairs \(\tup{p,\ddq}\) such that \(p\in\bbP_0\), \(\ddq\) is an \(\sS_0\)-name and \(\1_{\bbP_0}\forces\ddq\in\ddbbP_1\).\footnote{As in forcing iterations, since the collection of such names \(\ddq\) is technically a proper class, we would for instance put a bound on the rank of \(\ddq\).} The ordering is given by \(\tup{p',\ddq'}\leq\tup{p,\ddq}\) if \(p'\leq_{\bbP_0}p\) and \(p'\forces\ddq'\leq_{\ddbbP_1}\ddq\).

For \(\tup{p,\ddq}\in\bbP\), if \(\1_{\bbP_0}\forces\dot\pi_1\in\dsG_1\) then there is a canonical choice for an \(\sS_0\)\nobreakdash-name \(\ddq'\) such that \(\1_{\bbP_0}\forces\dot\pi_1(\ddq)=\ddq'\).\footnote{For instance, let \(\ddq'=\{\tup{p,\ddq}\mid\ddy\text{ is an }\sS_0\text{-name and }p\forces\ddy\in\dot\pi_1(\ddq)\}\), where we bound the rank of \(\ddy\) appropriately.} We shall denote this name simply by \(\dot\pi_1\ddq\) to coincide with the notation \(\pi_0p\) for the action of \(\sG_0\) on \(\bbP_0\).
We set \(\sG=\sG_0\ast\dsG_1\) to be the group with elements that are pairs \(\bar\pi=\tup{\pi_0,\dot\pi_1}\) where \(\pi_0\in\sG_0\) and \(\1_{\bbP_0}\forces\dot\pi_1\in\dsG_1\), with group action on \(\bbP\) given by
\begin{equation*}
\bar\pi\tup{p,\ddq}=\tup{\pi_0 p,\pi_0(\dot\pi_1\ddq)}=\tup{\pi_0 p,(\pi_0\dot\pi_1)(\pi_0\ddq)}.\footnote{Technically speaking this might not end up being an automorphism, since the action of \(\bar\pi\) might not be invertible as only very specific names are of the form \(\dot\pi_1\ddq_1\). There are several ways to work around this, though. For instance, this does define an automorphism in the separative quotient of \(\bbP\) (see \cite{Karagila2016}). Another work-around that will appear in an upcoming paper by the second author is to only consider names of the form \(\ddq=\id^\bullet\ddq\) in the definition of \(\bbP\) in the first place. Every \(\sS_0\)-name for an element of \(\ddbbP_1\) is forced to be equal to a name of this form, as can be easily seen.}
\end{equation*}
This action is sometimes denoted \(\mint_{\bar\pi}\tup{p,\ddq}\) in the literature, especially in longer iterations, and is used in Section~\ref{s:eg;ss:bristol}.

Whenever \(\Gamma_0\in\sF_0\), \(\1_{\bbP_0}\forces\dot\Gamma_1\in\dsF_1\) and \(\Gamma_0\leq\sym_{\sG_0}(\dot\Gamma_1)\), we identify \({\bar\Gamma=\tup{\Gamma_0,\dot\Gamma_1}}\) with the group of \(\tup{\pi_0,\dot\pi_1}\), where \(\pi_0\in\Gamma_0\) and \(\1_{\bbP_0}\forces\dot\pi_1\in\dot\Gamma_1\). Finally, we set \(\sF=\sF_0\ast\dsF_1\) to be the filter generated by subgroups of \(\sG\) of this form.

In our constructions we shall further require, without loss of generality, that each of \(\ddbbP_1\), \(\dot\leq_{\ddbbP_1}\), \(\dsG_1\), and \(\dsF_1\) are \emph{\(\sG_0\)-stable}. That is, \(\sG_0\leq\sym_{\bbP_0}(\ddbbP_1)\), etc.
\end{defn}

\subsection{Wreath products}

For this section only, let \(G\) and \(H\) be groups acting on sets \(X\) and \(Y\) respectively. We shall define the wreath product of \(G\) by \(H\) that will be required in Section~\ref{sec:Examples;subsec:morris-model} and is commonly used elsewhere in symmetric extensions.

\begin{defn}\label{defn:wreath-product}
Let \(G\) be a group acting on a set \(X\) and \(H\) a group acting on a set \(Y\).
\begin{itemize}
\item To each \(x\in X\) and \(h\in H\) we associate an element \(h^\ast_x\in S_{X\times Y}\) by 
\begin{equation*}
h^\ast_x(x',y)=\begin{cases}
(x',hy) & \text{if }x'=x,\\
(x',y) & \text{if }x'\neq x.
\end{cases}
\end{equation*}
\item To each \(g\in G\) we associate an element \(g^\ast\in S_{X\times Y}\) by
\begin{equation*}
g^*(x,y)=(gx,y). 
\end{equation*}
\end{itemize}
When then define the \emph{wreath product} of \(G\) by \(H\), denoted \(G\wr H\), is the subgroup of \(S_{X\times Y}\) consisting of elements of the form
\begin{equation*}
\pi=(g^*,(h^*_x)_{x\in X})\text{, where }\pi(x,y)=g^\ast(h_x^\ast(x,y)).
\end{equation*}
We denote by \(\pi^*\) the \(g^*\) component of \(\pi\) and by \(\pi_x^*\) the \(h_x^*\) component of \(\pi\).
\end{defn}

This notation differs from the typical presentation given in group-theoretic contexts, but it is more intuitive for our purposes. One may consider the action of \(G\wr H\) pictorially: Suppose that \(G\) acts on a circle \(X\) upon which we wind many smaller circles, each copies of \(Y\), to form a wreath. Then an element \((g^*,(h^*_x)_{x\in X})\) of \(G\wr H\) acts on the wreath by applying \(h^*_x\) to the copy of \(Y\) at the \(x\) co-ordinate, and then applying \(g^*\) to the wreaths as a collective, as in Figure~\ref{fig:wreath}.

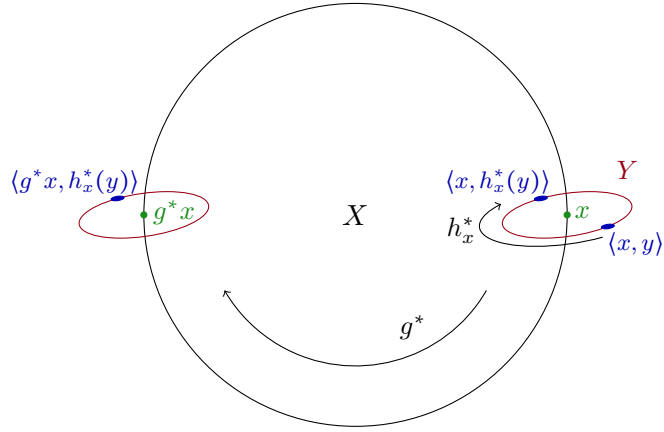
\begin{figure}[h]
\begin{tikzpicture}[scale=0.4]
	\draw[yslant=0,xslant=0] (0,0) circle [radius=7];
	\node (X) at (0,0) {\large$X$};
	\draw[->] ({sqrt(18.75)},{-2.5}) arc (-30:-150:5);
	\node (gast) at (1.9,-3.7) {$g^\ast$};
	\node (hast) at (3.5,-0.5) {$h_x^\ast$};

\begin{scope}[canvas is xz plane at y=0]
	\draw[color=fdred] (7,0) circle [radius=2];
	\draw[fill=fdblue,color=fdblue] ({7+sqrt(3)},{1}) circle [radius=0.2];
	\node[color=fdblue] (xy) at ({8.5+sqrt(3)},{1+1.5}) {\small$\tup{x,y}$};
	\draw[->] ({7.5+sqrt(2)},{0.5+sqrt(2)}) arc (45:200:{sqrt(4.5+sqrt(8))});
	\draw[fill=fdblue,color=fdblue] ({7-sqrt(2)},{-sqrt(2)}) circle [radius=0.2];
	\node[color=fdblue] (xhy) at ({5-sqrt(2)},{-1.5-sqrt(2)}) {\small$\tup{x,h_x^\ast(y)}$};

	\draw[color=fdred] (-7,0) circle [radius=2];
	\draw[fill=fdblue,color=fdblue] ({-7-sqrt(2)},{-sqrt(2)}) circle [radius=0.2];
	\node[color=fdblue] (gxhy) at ({-9-sqrt(2)},{-1.5-sqrt(2)}) {\small$\tup{g^\ast x,h_x^\ast(y)}$};
\end{scope}

	\node[color=fdred] (Y) at (9,1.5) {$Y$};
	\node[color=ffgreen] (x) at (7.5,0.1) {$x$};
	\draw[fill=ffgreen,color=ffgreen] (7,0) circle [radius=0.1];
	\node[color=ffgreen] (gx) at (-6,0.1) {$g^\ast x$};
	\draw[fill=ffgreen,color=ffgreen] (-7,0) circle [radius=0.1];

\end{tikzpicture}
\caption{The action of a wreath product: The co-ordinate \(\tup{x,y}\) is transformed by \(\pi=(g^\ast,(h_x^\ast)_{x\in X})\) first according to \(h_x^\ast\) into \(\tup{x,h_x^\ast(y)}\), and then according to \(g^\ast\) into \(\tup{g^\ast x,h_x^\ast(y)}\).}
\label{fig:wreath}
\end{figure}

\section{Upwards homogeneity}\label{sec:Upwards-Homogeneity}

Recall that we are concerned with iterations of symmetric extensions \(\sS=\sS_0\ast\dsS_1\) given by the symmetric systems \(\tup{\bbP_0,\sG_0,\sF_0}\ast\tup{\ddbbP_1,\dsG_1,\dsF_1}^\bullet=\tup{\bbP,\sG,\sF}\).

\begin{defn}\label{defn:upwards-homogeneous}
Let \(\sS=\tup{\bbP_0,\sG_0,\sF_0}\ast\tup{\ddbbP_1,\dsG_1,\dsF_1}^\bullet=\sS_0\ast\dsS_1\) be an iteration of symmetric extensions. We say that \(\sS\) is \emph{upwards homogeneous} if, for all \(\bar\Gamma\in\sF\), there is a predense collection of conditions \(\tup{p^\circ,\ddq^\circ}\in\bbP\) satisfying the following property:
\begin{quote}
For all \(\tup{p,\ddq},\tup{p,\ddq'}\leq\tup{p^\circ,\ddq^\circ}\) there is \(\bar\pi\in\bar\Gamma\) such that \(\bar\pi\tup{p,\ddq}\comp\tup{p,\ddq'}\).
\end{quote}
That is, there is a predense collection of elements \(\tup{p^\circ,\ddq^\circ}\) such that whenever \(p\leq p^\circ\), and \(p\forces_{\sS_0}\ddq,\ddq'\leq\ddq^\circ\) there is \(\pi_0\in\Gamma_0\) and \(\1_{\bbP_0}\forces_{\sS_0}\dot\pi_1\in\dot\Gamma_1\) such that \(\pi_0p\comp p\) and, for some \(p'\leq p,\pi_0(p)\), \(p'\forces_{\sS_0}\pi_0(\dot\pi_1\ddq)\comp\ddq'\).
\end{defn}

Note that if \(\tup{p^\circ,\ddq^\circ}\) satisfies this condition for a particular \(\bar\Gamma\in\sF\) then all extensions of \(\tup{p^\circ,\ddq^\circ}\) do as well. Hence, we may equivalently replace ``predense'' by dense or open dense in the definition.

\begin{mainthm}
\label{thm:upwards-homogeneous}
An iteration of symmetric extensions \(\sS=\sS_0\ast\dsS_1\) is upwards homogeneous if and only if for all symmetric extensions \(V\subseteq M\subseteq N\) that \(\sS\) generates and all \(X\in N\) with \(X\subseteq V\), we have \(X\in M\).
\end{mainthm}
Note that in the case \(V\vDash\AC\) this is equivalent to saying that there are no new sets of ordinals in \(N\) that were not already in \(M\).
\begin{proof}
\((\implies)\). Let \(\ddX\) be an \(\sS\)-name such that \(\1\forces_{\sS}\ddX\subseteq\check{A}\) for some \(A\in V\). Letting \(\bar\Gamma\in\sF\) be such that \(\bar\Gamma\leq\sym(\ddX)\), we claim that whenever \(\tup{p^\circ,\ddq^\circ}\in\bbP\) is as in the definition of upwards homogeneous, \(\tup{p,\ddq}\leq\tup{p^\circ,\ddq^\circ}\), and \(\tup{p,\ddq}\forces_{\sS}\check{a}\in\ddX\), then \(\tup{p,\ddq^\circ}\forces_{\sS}\check{a}\in\ddX\) as well.

Let \(\tup{p',\ddq'}\leq\tup{p,\ddq^\circ}\). By the definition of upwards homogeneous, there is \(\bar\pi\in\bar\Gamma\) such that \(\bar\pi\tup{p',\ddq}\comp\tup{p',\ddq'}\). Given that \(\tup{p',\ddq}\leq\tup{p,\ddq}\), we have that \({\tup{p',\ddq}\forces_{\sS}\check{a}\in\ddX}\) and, indeed, \(\bar\pi\tup{p',\ddq}\forces_{\sS}\check{a}\in\ddX\) as well. Since \(\bar\pi\tup{p',\ddq}\comp\tup{p',\ddq'}\), we cannot have that \(\tup{p',\ddq'}\forces_{\sS}\check{a}\notin\ddX\) and so \(\tup{p,\ddq^\circ}\forces_{\sS}\check{a}\in\ddX\) as required. Similarly, if \(\tup{p,\ddq}\forces_{\sS}\check{a}\notin\ddX\), then \(\tup{p,\ddq^\circ}\forces_{\sS}\check{a}\notin\ddX\) as well. Hence, if we let
\begin{equation*}
\ddY=\left\{\tup{\pi p,\check{a}}\SetSymbol \pi\in\Gamma_0,\,a\in A,\,\tup{\pi p,\pi\ddq^\circ}\forces_{\sS}\check{a}\in\ddX\right\}
\end{equation*}
then \(\Gamma_0\leq\sym_{\sG_0}(\ddY)\), so \(\ddY\) is an \(\sS_0\)-name, and \(\tup{p^\circ,\ddq^\circ}\forces_{\sS}\ddX=\ddY\).

\((\impliedby)\). Let \(\bar\Gamma\) and \(\tup{p_0,\ddq_0}\) be arbitrary. Consider the \(\bbP\)-name
\begin{equation*}
\ddX=\left\{\tup{\bar\pi\tup{p,\ddq},\sidecheck{\tup{p,\ddq}}}\SetSymbol \tup{p,\ddq}\in\bbP,\,\bar\pi\in\bar\Gamma\right\},
\end{equation*}
where \(\sidecheck{\tup{p,\ddq}}\) refers to the check name for the \emph{ground model set} \(\tup{p,\ddq}\), rather than the condition in \(\bbP\) or the \(\sS_0\)-name for an element of \(\ddbbP_1\). Note that \(\bar\Gamma\leq\sym(\ddX)\) and every \(\bbP\)-name appearing in \(\ddX\) is a check name, so \(\ddX\) is a hereditarily \(\sF\)-symmetric name for a subset of \(V\). Note also that since all \(\bbP\)-names appearing in \(\ddX\) are check names, for all generic \(G\subseteq\bbP\) we have \(\tup{p,\ddq}\in\ddX^G\) if and only if there is \(\bar\pi\in\bar\Gamma\) such that \(\bar\pi\tup{p,\ddq}\in G\). By assumption there is an \(\sS_0\)-name \(\ddY\) and \(\tup{p^\circ,\ddq^\circ}\leq \tup{p_0, \ddq_0}\) such that \({\tup{p^\circ,\ddq^\circ}\forces_{\sS}\ddX=\ddY}\).

We claim that \(\tup{p^\circ,\ddq^\circ}\) is as in the definition of upwards homogeneity. Towards this end, let \(\tup{p,\ddq},\tup{p,\ddq'}\leq\tup{p^\circ,\ddq^\circ}\) be arbitrary. Since \(\tup{p,\ddq}\forces_{\sS}\sidecheck{\tup{p,\ddq}}\in\ddX\) and \(\tup{p,\ddq}\leq\tup{p^\circ,\ddq^\circ}\), we have that \(\tup{p,\ddq}\forces_{\sS}\sidecheck{\tup{p,\ddq}}\in\ddY\) as well. Furthermore, since \(\ddY\) and \(\sidecheck{\tup{p,\ddq}}\) are \(\sS_0\)-names, we in fact have that \(p\forces_{\sS_0}\sidecheck{\tup{p,\ddq}}\in\ddY\). If \(G\subseteq\bbP\) is then an arbitrary generic filter with \(\tup{p,\ddq'}\in G\) we have that \(\tup{p,\ddq}\in\ddY^G=\ddX^G\) and, by the observation above, there is \(\bar\pi\in\bar\Gamma\) such that \(\bar\pi\tup{p,\ddq}\in G\). Since both \(\tup{p,\ddq'}\) and \(\bar\pi\tup{p,\ddq}\) lie in the same filter, these conditions must be compatible and we have proved our claim.
\end{proof}

By inspecting the prior proof, we have actually made a somewhat stronger assertion.

\begin{thm}
Let \(\bar\Gamma\in\sF\) and \(\tup{p^\circ,\ddq^\circ}\in\bbP\). The following are equivalent:
\begin{enumerate}
\item For all \(\tup{p,\ddq},\tup{p,\ddq'}\leq\tup{p^\circ,\ddq^\circ}\) there is \(\bar\pi\in\bar\Gamma\) such that \(\bar\pi\tup{p,\ddq}\comp\tup{p,\ddq'}\);
\item \(\tup{p^\circ,\ddq^\circ}\) forces that all \(\bar\Gamma\)-stable names are equivalent to an \(\sS_0\)-name; and
\item \(\tup{p^\circ,\ddq^\circ}\) forces that the \emph{\(\bar\Gamma\)-symmetrisation} of the generic object \(G\),
\begin{equation*}
\bar\Gamma\ddG=\left\{\bar\pi\tup{\tup{p,\ddq},\sidecheck{\tup{p,\ddq}}}\SetSymbol\bar\pi\in\bar\Gamma,\,\tup{p,\ddq}\in\bbP\right\}
\end{equation*}
is equivalent to an \(\sS_0\)-name. \qed
\end{enumerate}
\end{thm}

For the rest of the paper we shall exhibit how this definition uniformises various constructions that use upwards homogeneity.

\section{Examples}\label{sec:Examples}

\subsection{Monro's iteration}\label{sec:Examples;subsec:monros-iteration}

Monro's article \cite{Monro1973} was inspired by a theorem of Balcar and Vop\v{e}nka (from \cite{vapenka_complete_1967}) that if \(M\) and \(N\) are transitive models of \(\ZF\) such that \(\power(\Ord)^M=\power(\Ord)^N\) and either \(M\vDash\AC\) or \(N\vDash\AC\), then \(M=N\). Monro's paper also built on a later construction by Jech in \cite{jech_models_1971} of a counterexample to this theorem for general models of \(\ZF\). His construction extends this process to build a chain of models \(\tup{M_n\mid n<\omega}\) of \(\ZF\), where \(M_0\) is a model of \(\ZFC\), such that for all \(n<\omega\), \({\power^n(\Ord)^{M_n}=\power^n(\Ord)^{M_{n+1}}}\), but \(M_n\neq M_{n+1}\), where \(\power^n(\Ord)\) refers to the \(n\)\textup{th} iterate of the power set operation on \(\Ord\). This construction would later be extended even further by \cite{Shani2018} to go beyond the \(\omega\)\textup{th} step.

We shall inductively define a sequence of models \(\tup{M_n\mid n<\omega}\), Dedekind\nobreakdash-finite sets \(A_n\in M_n\) (excepting \(A_0=\omega\)), and symmetric systems \({\sS_n=\tup{\bbP_n,\sG_n,\sF_n}\in M_n}\) such that \(M_{n+1}\) is a symmetric extension of \(M_n\) by \(\sS_n\). We begin with \(M_0=V\) a model of \(\ZFC\) and let \(A_0=\omega\). If \(M_n\) and \(A_n\) have been chosen, the symmetric system \(\sS_n\) is defined as follows.

The forcing \(\bbP_n\) is \(\Add(A_n,\omega)\). Note that since \(A_n\) is \(\omega\) or Dedekind-finite, this is the forcing given by finite partial functions \(\omega\times A_n\to2\) ordered by reverse inclusion. We define the \(\bbP_n\)-names
\begin{align*}
\dda^n_k & =\{\tup{p,\check{x}}\mid p(k,x)=1,\ x\in A_n\}\text{ and}\\
\ddA_{n+1} & =\{\dda^n_k\mid k<\omega\}^\bullet.
\end{align*}
Let \(\sG_n= S_{\lomega}\), the group of finitary permutations of \(\omega\) with group action on \(\bbP_n\) given by \({\pi p(\pi n,x)=p(n,x)}\). Then \(\pi\dot{a}_k^n=\dot{a}_{\pi k}^n\) and \(\pi \dot{A}_{n+1}=\dot{A}_{n+1}\) for all \(\pi\in\sG_n\).

The filter \(\sF_n\) is generated by the groups \(\fix(E)\defeq\{\pi\in\sG_n\mid\pi\res E=\id\}\) for \(E\in [\omega]^\lomega\). By our prior calculations we get that \(\fix(\{k\})\leq\sym_{\bbP_n}(\dda^n_k)\in\sF_n\) and \(\sG_n=\sym_{\bbP_n}(\ddA_{n+1})\in\sF_n\), so the names \(\dda_k^n\) and \(\ddA_{n+1}\) are all in \(\HS_{\bbP_n}\). Finally, we set \(M_{n+1}\) to be the symmetric extension over \(M_n\) given by the symmetric system \(\sS_n=\tup{\bbP_n,\sG_n,\sF_n}\).

In this case \(M_1\) is Cohen's first model, defined in Example~\ref{eg:cohens-first-model}, and by the same argument the sets \(A_{n+1}\) are all infinite Dedekind-finite.

\begin{prop}\label{prop:monro}
For all \(n<\omega\), \(\sS_n\ast\dsS_{n+1}\) is upwards homogeneous. Therefore, \(M_{n+2}\) does not have any subsets of \(M_n\) that were not already in \(M_{n+1}\) and, by induction, \(\power^{n+1}(\Ord)^{M_{n+1}}=\power^{n+1}(\Ord)^{M_{n+2}}\). 
\end{prop}
\begin{proof}
Here we work in \(M_n\), so \(\sS_n\) is an element of our ground model \(M_n\) and the statement ``\(\sS_n\ast\dsS_{n+1}\) is upwards homogeneous'' makes sense. Note that conditions in \(\ddbbP_{n+1}\) have canonical \(\bbP_n\)-names: For \(f\) a finite partial function \(\omega\times\omega\to 2\) we define the \(\bbP_n\)-name \(\ddq_f\) for a condition in \(\ddbbP_{n+1}\) as
\begin{equation*}
\ddq_f=\{\langle \check{k},\dot{a}_m^n,\sidecheck{(f(k,m))}\rangle^\bullet\mid \langle k,m\rangle\in \dom(f)\}^\bullet.
\end{equation*}
That is, \(q_f(k,a^n_m)=f(k,m)\). Conversely, whenever \(p\forces_{\sS_n}\ddq\in\ddbbP_{n+1}\) we may extend \(p\) to \(p'\leq p\) and find some \(f\colon\omega\times \omega\to 2\) such that \(p'\forces_{\sS_n}\ddq=\ddq_f\). When the conditions of \(\ddbbP_{n+1}\) are cast into this canonical form, elements \(\pi\in\sG_n\) act on them as follows: Let \(\hat\pi\cdot f\) denote the function such that \((\hat\pi\cdot f)(k,\pi m)=f(k,m)\) for all \(k,m<\omega\). Then 
\begin{align*}
\pi\ddq_f &= \{ \langle \pi\check{k},\pi\dot{a}_m^n,\sidecheck{\pi(f(k,m))}\rangle^\bullet\mid \langle k,m\rangle\in \dom(f)\}^\bullet \\
& = \{ \langle \check{k},\dot{a}_{\pi m}^n,\sidecheck{(f(k,m))}\rangle^\bullet\mid \langle k,m\rangle\in \dom(f)\}^\bullet \\
& = \{ \langle \check{k},\dot{a}_{\pi m}^n,\sidecheck{(\hat\pi\cdot f(k,\pi m))}\rangle^\bullet\mid \langle k,m\rangle\in \dom(f)\}^\bullet \\
& = \{ \langle \check{k},\dot{a}_m^n,\sidecheck{(\hat\pi\cdot f(k,m))}\rangle^\bullet\mid \langle k,m\rangle\in \dom(\hat\pi\cdot f)\}^\bullet \\
& = \dot{q}_{\hat\pi\cdot f}
\end{align*}
We shall still denote by \(\pi f\) the action \(\pi f(\pi n,m)=f(n,m)\). In the following, we assume that all conditions in \(\ddbbP_{n+1}\) are in this canonical form.

Let \(\fix(E_0)\in\sF_n\) and \(\fix(E_1)\in\dsF_{n+1}\), where \(E_0,\,E_1\in[\omega]^\lomega\), and let \({\tup{p_0,\ddq_f}}\) be a condition in \({\bbP_n\ast\ddbbP_{n+1}}\). We wish to find a compatible condition \(\tup{p^\circ,\ddq^\circ}\) that witnesses upwards homogeneity for \(\fix(E_0)\ast\fix(E_1)\in\sF_n\ast\dsF_{n+1}\). We claim for any \(f^\circ\) such that \(\dom(f^\circ)=E_1\times E_0\) and \(f^\circ\comp f\), \(\tup{p_0,\ddq_{f^\circ}}\) is sufficient.\footnote{For example, \(f^\circ=f\res E_1\times E_0\cup\{\tup{a,0}\mid a\in (E_1\times E_0)\backslash\dom(f)\}\).} Set \(\ddq^\circ=\ddq_{f^\circ}\). Note that in this case we certainly have that \(\tup{p_0,\ddq_f}\comp\tup{p_0,\ddq^\circ}\).

Let \(\tup{p,\ddq_g},\tup{p,\ddq_h}\leq\tup{p_0,\ddq^\circ}\). We can find the required \(\pi\in\fix(E_0)\) and \(\dot\sigma\in\fix(E_1)\) as follows:

Firstly, since \(\dom(g)\) and \(\dom(h)\) are finite, there is \(\sigma\in\fix(E_1)\) such that \(\dom(\sigma g)\cap\dom(h)\subseteq E_1\times\omega\), akin to Cohen's first model in Example~\ref{eg:cohens-first-model}.

Similarly, given that \(\dom(p)\), \(\dom(\sigma g)\), and \(\dom(h)\) are all finite we can find \(\pi\in\fix(E_0)\) such that \(\dom(\pi p)\cap\dom(p)\subseteq E_0\times A_n\) and
\begin{equation*}
\dom(\hat\pi\cdot(\sigma g))\cap\dom(h)\subseteq E_1\times E_0.
\end{equation*}
Given that \(g,\,h\supseteq f^\circ\), and \(\dom(f^\circ)=E_1\times E_0\), we must have that \(\hat\pi\cdot(\sigma g)\comp h\), and hence \(\tup{\pi,\dot\sigma}\tup{p,\ddq_g}\comp\tup{p,\ddq_h}\) as required.
\end{proof}

\begin{rk}
In the proof of Proposition~\ref{prop:monro} we used the fact that there exist canonical \(\sS_0\)-names for conditions in \(\ddbbP_1\) on which the group \(\sG_0\) acts nicely. This is the case for all applications of upwards homogeneity that we will exhibit and, for the rest of the paper, we shall implicitly assume that the names for the conditions of \(\ddbbP_1\) are already cast into such canonical forms. Extending the first co-ordinate to achieve this will not affect whether \(\sS\) is upwards homogeneous.
\end{rk}

This case serves as an excellent intuitive framework for upwards homogeneity. If one only considers \(\fix(E_0)\leq\sym_{\sG_0}(\dot{X})\) and \({\fix(E_1)\leq\sym_{\dsG_1}([\dot{X}])}\), where \([\dot{X}]\) is an \(\sS_0\)-name for an \(\dsS_1\)-name \(\dot{X}\) of a set \(X\), then any \(\ddq\in\ddbbP_1\) that has information filled up on the domain \(E_0\times E_1\) determines the whole content of \(X\). Since \(E_0\) and \(E_1\) are both finite, we can always find some \(\ddq\) such that \(\dom(\ddq)\supset E_0\times E_1\), from which we can determine \(X\).

In this way, one may view upwards homogeneity as a sort of ``closure'' of \(\ddbbP_1\) with respect to both filters \(\sF_0\) and \(\dsF_1\). In many cases of symmetric extensions the filter is generated by \({\{\fix(E)\mid E\in\calI}\}\), where \(\calI\) is an ideal of subsets of a set related to the forcing. If \(\ddbbP_1\) is not closed enough compared to one of the ideals that are used to generate the filters then upwards homogeneity fails.

\begin{eg}\label{eg:counterexamples}
Let \(\bbP_0=\Add(\omega,\omega_1)\), \(\sG_0=S_{\omega_1}\) acting on \(\bbP_0\) via \(\pi p(\pi(\alpha),n)=p(\alpha,n)\), and \(\sF_0\) be the filter generated by \(\{\fix(E)\mid E\in [\omega_1]^{{<}\omega_1}\}\). Then \(\bbP_0\) adds a set \(A=\{a_\alpha\mid\alpha<\omega_1\}\) of \(\omega_1\)-many new reals. If \(E\subset\omega_1\) is countable in \(V\) then there is an enumeration \(\tup{a_\alpha\mid \alpha\in E}\) in \(M\), since it is fixed by \(\fix(E)\).

Then we let \(\ddbbP_1\) be the notion of forcing with conditions that are finite partial functions \(q\colon\omega\times A\to2\), ordered by reverse inclusion. Let \(\dsG_1=S_\omega\), acting on \(\ddbbP_1\) via \(\dot\sigma\ddq(\dot\sigma(\check{n}),\dda)=\ddq(\check{n},\dda)\), and \(\dsF_1\) be the filter generated by \(\{\fix(E)\mid E\in [\omega]^\lomega\}\). Then \(\ddbbP_1\) adds a set \(\{b_n\mid n<\omega\}\) where each \(b_n\) is a subset of \(A\). Let \(E\subset \omega_1\) be countable, then for each \(n<\omega\), \(\ddbbP_1\) adds \(\{\alpha\in E\mid a_\alpha\in b_n\}\) new to \(M\), since for each \(r\colon E\to 2\) in \(M\), \(\{ q\in\ddbbP_1\mid (\exists\alpha)q(a_\alpha,n)\neq r(n)\}\) is dense.
\end{eg}

In this example \(\ddbbP_1\), consisting of finite conditions, is not closed enough compared to the \(\sigma\)-closedness of the filter \(\sF_0\).

The reader may also verify that if in Monro's iteration \(\sS_0\ast\dsS_1\) we change the filter of either the first or the second iteration to the improper filter (the filter that contains the trivial subgroup) then we have added new subsets of \(\omega\) in the second iteration.

\subsection{Forcing over symmetric extensions}\label{sec:Examples;subsec:second-iterand-full}

Here we consider what happens to Definition~\ref{defn:upwards-homogeneous} if the second iterand is equivalent to a forcing notion. In this case, one may take \(\dsG_1=\{\id\}\) and upwards homogeneity can be simplified: For all \(\Gamma\in\sF_0\) there is a predense collection of \(\tup{p^\circ,\ddq^\circ}\in\bbP\) such that, for all \(\tup{p,\ddq},\,\tup{p,\ddq'}\leq\tup{p^\circ,\ddq^\circ}\), there is \(\pi\in\Gamma\) such that \(\tup{\pi p,\pi \ddq}\comp\tup{p,\ddq'}\). Notationally, we shall identify a (name for a) notion of forcing \(\ddbbP_1\) with the symmetric system \(\tup{\ddbbP_1,\{\id\},\{\{\id\}\}}^\bullet\).

An example of this scenario is \cite{KaragilaSchlicht2019}. Recall Cohen's first model defined in Example~\ref{eg:cohens-first-model} and let \(A=\{ a_n\mid n<\omega\}\) denote the Dedekind-finite set of reals added by \(\sS_0\) in this extension.

\begin{prop}[{\cite[Theorem~5.1]{KaragilaSchlicht2019}}]\label{prop:second-iterand-eg}
Let \(\kappa\) be any cardinal and let \(\bbP_1\) be \(\Col(A,\kappa)\), the notion of forcing given by well-orderable partial functions \(q\colon A\to\kappa\) with \(\abs{q}<\abs{A}\), ordered by reverse inclusion. Then \(\sS_0\ast\Col(\ddA,\check{\kappa})^\bullet\) is upwards homogeneous. That is, \(\bbP_1\) does not add new sets of ordinals.
\end{prop}
\begin{proof}
Note that the conditions in \(\bbP_1\) are finite. For \(\tup{p_0,\ddq_0}\in\ddbbP_1\) and \(E\in[\omega]^\lomega\), we find \(\tup{p^\circ,\ddq^\circ}\leq\tup{p_0,\ddq_0}\) so that \(p^\circ\forces_{\sS_0}\dom(\ddq^\circ)\supseteq\{\dda_n\mid n\in E\}^\bullet\). If \(\tup{p,\ddq}\) and \(\tup{p,\ddq'}\) extend \(\tup{p^\circ,\ddq^\circ}\), let \(p'\leq p\) so that \(p'\forces\dom(\ddq)=\{\dda_n\mid n\in F\}^\bullet\land\dom(\ddq')=\{\dda_n\mid n\in F'\}^\bullet\) for some \(F, F' \in [\omega]^\lomega\). Then we find \(\pi\in\fix(E)\) such that \({\dom(\pi p')\cap\dom(p')\subseteq E}\) and \(\pi`` F \cap F' = E\) and thus \( p'\forces_{\sS_0}\dom(\pi\ddq)\cap\dom(\ddq')=\{\dda_n\mid n\in E\}^\bullet\). It follows that \(\pi\tup{p,\ddq}\comp\tup{p,\ddq'}\). Therefore \(\tup{p^\circ, \ddq^\circ}\) witnesses upwards homogeneity as required.
\end{proof}

\begin{rk}
Note that if instead the first iterand is a full forcing extension then \(\bbP_0\ast\dsS_1\) is upwards homogeneous if and only if \(\1_{\bbP_0}\) forces that \(\dsS_1\) is trivial. In the case that the ground model satisfies \(\AC\) this is an immediate consequence of Balcar and Vop\v{e}nka's theorem.
\end{rk}

\subsection{Morris's iteration}\label{sec:Examples;subsec:morris-model}

A remarkable construction that violates choice is showing that countable unions of countable sets need not be countable. For example, in \cite{feferman_independence_1963} Feferman and L\'{e}vy construct a model of \(\ZF\) in which \(2^\omega\) is a countable union of countable sets (and consequently \(\omega_1\) is singular). In the same vein, as part of their PhD thesis \cite{Morris1970}, Morris constructed a model of \(\ZF\) in which, for all ordinals \(\alpha\), there is a set \(A_\alpha\) that is a countable union of countable sets admitting a surjection \(\power(A_\alpha)\to\aleph_\alpha\). The construction is recast into a modern form in \cite{Karagila2020} and corrected in \cite{Karagila2021}, the method of which we follow here.

Fix an infinite cardinal \(\kappa\) and define \(\bbP_0=\Add(\kappa,\omega\times\omega\times\kappa)\). Let
\begin{align*}
\ddx_{n,m,\alpha}&=\{\tup{p,\check{\beta}}\mid p(n,m,\alpha,\beta)=1\},\\
\dda_{n,m}&=\{\ddx_{n,m,\alpha}\mid\alpha<\kappa\}^\bullet,\text{ and}\\
\ddA_n&=\{\dda_{n,m}\mid m<\omega\}^\bullet.
\end{align*}
Set \(\sG_0=(\{\id\}\wr S_\omega)\wr S_\kappa\leq S_{\omega\times\omega\times \kappa}\), so elements are of the form
\begin{equation*}
\pi(n,m,\alpha)=\tup{n,\pi^\ast_n(m),\pi_{n,m}(\alpha)}
\end{equation*}
as in Definition~\ref{defn:wreath-product}, with action given by \({\pi p(\pi(n,m,\alpha),\beta)=p(n,m,\alpha,\beta)}\). Hence \(\pi\ddx_{n,m,\alpha}=\ddx_{\pi(n,m,\alpha)}\), \(\pi\dda_{n,m}=\dda_{n,\pi^\ast_n(m)}\), and \(\pi\ddA_n=\ddA_n\).

For \(n<\omega\) and \(E\in[\kappa]^{{<}\kappa}\), let \(\fix(n,E)=\{\pi\in\sG_0\mid\pi\res n\times\omega\times E=\id\}\), and set \(\sF_0\) to be the filter of subgroups generated by these groups.

Let \(M\) be a symmetric extension by the system \(\sS_0=\tup{\bbP_0,\sG_0,\sF_0}\). Working in \(M\), we define \(\bbP_1\): conditions in \(\bbP_1\) are finite sequences \(t=\tup{t_i\mid i\in E}\), where \(E\in[\kappa\times\omega]^\lomega\) and, for some \(n<\omega\), \(t_i\in \prod_{k<n}A_k\) for all \(i\in E\). We define \(\supp(t)=\dom(t)=E\), the \emph{support} of this condition. We say that \(t\leq s\) if:
\begin{enumerate}
\item \(\supp(t)\supseteq\supp(s)\);
\item for all \(i\in\supp(s)\), \(s_i\subseteq t_i\); and
\item for all \(i\neq j\in\supp(s)\), if \(t_i(k)=t_j(k)\) then \(k\in\vert s_i \vert (= \vert s_j \vert )\).
\end{enumerate}
That is, if \(t\) extends the sequences \(s_i\) in any way then these extensions must be pairwise distinct. We define the group \(\sG_1\) as \(\{\id\}\wr S_\lomega\leq S_{\kappa\times\omega}\) (where \(S_\lomega\) is the set of finite-support elements of \(S_\omega\)), acting on \(\bbP_1\) via \(\pi\tup{t_i\mid i\in E}=\tup{t_{\pi i}\mid i\in E}\). Finally, we let \(\sF_1\) be the filter of subgroups of \(\sG_1\) generated by groups of the form \(\fix(E)=\{\pi\in\sG_1\mid\pi\res E=\id\}\) for \(E\in[\kappa\times\omega]^\lomega\). Let \(\sS_1=\tup{\bbP_1,\sG_1,\sF_1}\), and let \(\dsS_1=\tup{\ddbbP_1,\dsG_1,\dsF_1}^\bullet\) be a canonical \(\sS_0\)-name for \(\sS_1\).

Let \(V \subseteq M \subseteq N\) be an iterated symmetric extension according to \(\sS_0\ast\dsS_1\). In \(M\), the sets \(A_n\) are all countable, as is the set \(\{A_n \mid n < \omega\}\), so \(A = \{ a_{n,m} \mid n,m<\omega\}\) is a countable union of countable sets. Furthermore, a standard argument shows that \(A\) is not countable. Also in \(M\), let \(\ddb_{\alpha,n}=\{\tup{t,\check{a}}\mid(\exists i)t_{\alpha,n}(i)=a\}\), \(\ddB_\alpha=\{\ddb_{\alpha,n}\mid n < \omega\}^\bullet\), and \(\ddB = \{\ddb_{\alpha,n} \mid\alpha<\kappa, n<\omega\}^\bullet\), noting that these are all hereditarily symmetric. In \(N\), \(B \subseteq A\) and the canonical enumeration \(\tup{B_\alpha \mid \alpha < \kappa}\) induces a surjection \(\power(A) \to B \to \kappa\). Crucially, Proposition~\ref{prop:morris-upwards-homogeneous} demonstrates that \(N\) contains no sets of ordinals not already found in \(M\) which, in particular, implies that \(\kappa\) has not been collapsed. By iterating such constructions and using upwards homogeneity, one can continue to guarantee that no cardinals are collapsed while adding the desired sets \(A_\alpha\) for all \(\alpha\).

\begin{prop}\label{prop:morris-upwards-homogeneous}
\(\sS_0\ast\dsS_1\) is upwards homogeneous.
\end{prop}

\begin{proof}
Let \(\tup{\Gamma_0,\dot\Gamma_1}\in\sF\) and \(\tup{p_0,\ddt_0}\in\bbP\). Then we find \(\tup{p^\circ,\ddt^\circ} \leq \tup{p_0,\ddt_0}\), \(n<\omega\), \(E_0 \in [\kappa]^{{<}\kappa}\) and \(E_1 \in [\kappa\times\omega]^\lomega\) such that \(\fix(n, E_0) \leq \Gamma_0\) and
\begin{equation*}
p^\circ \forces_{\sS_0} \fix(\check{E}_1) \leq \dot \Gamma_1 \land \supp(\dot t^\circ) = \check E_1 \land (\forall i \in \check{E}_1) (\lvert \ddt^\circ_i \rvert = \check{n}). 
\end{equation*}
Suppose that \(\tup{p,\ddt},\,\tup{p,\ddt'}\leq\tup{p^\circ,\ddt^\circ}\). Then we may find \(p' \leq p\), \(F, F' \in [\kappa\times\omega]^\lomega\) extending \(E_1\), and \(f \colon F \to \omega^m\), \(f' \colon F' \to \omega^{m'}\), for some \(m\), \(m'\) such that
\begin{equation*}
p'\forces_{\sS_0} \ddt_i = \tup{\dda_{k, f_i(k)} \mid k < m}^\bullet \land \ddt'_j = \tup{\dda_{k, f'_j(k)} \mid k < m'}^\bullet,
\end{equation*}
for all \(i \in F\) and \(j \in F'\). We first may find \(\sigma\in\fix(E_1)\) such that \(\sigma`` F \cap F' = E_1\) and thus \(p' \forces \supp(\check\sigma \ddt) \cap \supp (\ddt') = \check{E}_1\). Hence the only reason that \(\check\sigma \ddt\) and \(\ddt'\) might be incompatible is because of their values on \(E_1\). Note that for all \(k \in [n, m)\) and \(i \neq j \in E_1\) we have \(f_i(k) \neq f_j(k)\), and a similar argument holds for \(f'\). Thus \(\{ \tup{f_i(k), f'_i(k)} \mid i \in E_1 \}\) is an injective function and can be extended to \(b_k \in S_\omega\).  Let \(\pi \in \fix(n, E_0)\) be such that for all \(k\in[n,m)\), \(\pi^*_k = b_k\), taking \(m\leq m'\) without loss of generality. By using the last factor of the wreath product which acts on \(\kappa\) we may also ensure that \(\pi p' \comp p'\). Thus, if \(r \leq \pi p', p'\),
\begin{align*}
r \forces (\pi\check \sigma \ddt)_i 
	&= \pi \tup{\dda_{k, f_i(k)} \mid k < m}^\bullet \\
	&= \tup{ \pi \dda_{k, f_i(k)} \mid k < m}^\bullet \\
	&= \tup{\dda_{k, f'_i(k)} \mid k < m}^\bullet \\
	&= \ddt'_i\res m,
\end{align*}
so \(r \forces_{\sS_0} \pi\check \sigma\ddt \comp \ddt'\). Hence also \(\tup{\pi,\check\sigma}\tup{p,\ddt}\comp\tup{p,\ddt'}\). Note that while it is not necessarily the case that \(\1_{\bbP_0} \forces_{\sS_0} \check{\sigma} \in \dot \Gamma_1\), we can easily find an \(\sS_0\)-name \(\dot \tau\) such that \(\1_{\bbP_0} \forces_{\sS_0} \dot \tau \in \dot \Gamma_1\) and \(p^\circ \forces_{\sS_0} \dot \tau = \check{\sigma}\). Then \(\tup{\pi, \dot \tau}\) is as required.
\end{proof}

\subsection{The Bristol model}\label{s:eg;ss:bristol}

In \cite{vopenka_semisets_1972} Vop\v{e}nka shows that for all set-generic extensions \(V\subseteq V[G]\) of models of \(\ZFC\) and intermediate models \(V\subseteq W\subseteq V[G]\), if \(W\vDash\ZFC\) then \(W\) is also a set-generic extension of \(V\).\footnote{More precisely they show that \(L[A]\) is a generic extension of \(\operatorname{HOD}\) whenever \(A\) is a set of ordinals (see \cite[Theorem~15.46]{Jech1973}), but the theorem generalises easily (see \cite[Lemma~15.43]{Jech1973}).} However, if we ask only that \(W\) is a model of \(\ZF\) then this principle no longer holds. The Bristol Model is a model \(M\vDash\ZF\) such that \(L\subseteq M\subseteq L[\rho_0]\) but \(M\) fails to be a symmetric extension of \(L\), where \(\rho_0\) is Cohen\nobreakdash-generic over \(L\). This was first exposited in \cite{Karagila2018} and later expanded upon in \cite{Karagila2020ABM}. While we will not reproduce the entire construction of the Bristol model here, we will describe an upwards homogeneity argument made for limit iterands of the construction.

Originally, the precise question that the Bristol model sought to answer was ``if \(V \subseteq M \subseteq V[G]\), with \(V\) a model of \(\ZFC\) and \(M\) a model of \(\ZF\), must \(M\) be of the form \(V(x)\) for some \(x \in V[G]\)?''\footnote{Usuba \cite{usuba_choiceless_2021} later proved that this question is equivalent to asking if all intermediate models are symmetric extensions, per the prior paragraph.} This was answered in the affirmative for symmetric extensions by Grigorieff \cite{grigorieff_intermediate_1975}, but the question remained open in generality. The Bristol model answered this question negatively using upwards homogeneity. The model is a union of intermediate symmetric extensions \(V \subseteq M_0 \subseteq \cdots \subseteq M_\alpha \cdots \subseteq V[\rho_0]\) such that the \(\alpha\)\textup{th} stage violates the \(\alpha\)\textup{th} \emph{Kinna--Wagner principle}, which states ``for all \(X\) there is an ordinal \(\beta\) such that there is an injection \(X \to \power^\alpha(\beta)\).'' By using upwards homogeneity, the violations of these principles would be maintained and, in the final model, all Kinna--Wagner principles are violated. A simple argument shows that if \(V\) is a model of \(\ZFC\) and \(x\) has von Neumann rank \(\alpha\) then \(V(x)\) is a model of the \(\alpha\)\textup{th} Kinna--Wagner principle, so the Bristol model is indeed not a model of \(V(x)\) for any \(x \in V[\rho_0]\).

All the definitions here are made in \(L\). If \(\vec{a}\) is a sequence indexed by ordinals, \(a_\beta\) denotes its \(\beta\)\textup{th} component. The Bristol model is a class-length iteration of symmetric extensions beginning at \(L\). To define this sequence we require, for each \(\beta\in\Ord\), a name \(\dot{\rho}_\beta\). For successor \(\beta\), \(\dot{\rho}_\beta\) is a name for a sequence of length \(\omega_\beta\). We will not define \(\dot\rho_\beta\) in this case, but we will define \(\dot\rho_\beta\) for limit \(\beta\) when it is due. The \(\beta\)\textup{th} iterand is given by the symmetric system denoted \(\dot\sI_\beta=\tup{\ddbbQ_\beta,\dsG_\beta,\dsF_\beta}\), and the iterated symmetric system of the first \(\beta\) iterands is denoted \(\sS_\beta=\tup{\bbP_\beta,\frakG_\beta,\frakF_\beta}\).

Throughout this section we use the notation \(\mint_{\bar\pi}\) to mean the action of \(\bar\pi\) on elements of an iteration. In the case of a two-step iteration we define this by \(\mint_{\tup{\pi_0,\dot\pi_1}}\tup{p,\ddq}=\tup{\pi_0p,\pi_0(\dot\pi_1\ddq)}\) as usual, and in the case that we have a finite-support iteration we appeal to the fact that for all \(\bar\pi=\tup{\dot\pi_\alpha\mid\alpha<\gamma}\) there will be finitely many \(\alpha<\gamma\) such that \(\dot\pi_\alpha\) is non-trivial.

\begin{fact}[{Part of the induction hypothesis in \cite[\S4.2]{Karagila2018}}]
\label{fact:homogeneity-qgamma}
Suppose that we have constructed \(\dot\sI_\gamma=\tup{\ddbbQ_\gamma,\dsG_\gamma,\dsF_\gamma}^\bullet\) for all \(\gamma<\beta\). Then for all \(\gamma<\beta\) and \(\ddq,\ddq'\in\ddbbQ_\gamma\) there is some \(\dot\pi\in\dsG_\gamma\) such that \(\dot\pi \ddq\comp \ddq'\).
\end{fact}

\subsubsection{Iterating the systems \(\dot\sI_\gamma\)} Note that if we have constructed \(\dot\sI_\gamma\) for all \(\gamma<\beta\) then we may combine these symmetric systems into a single symmetric system \(\sS_\beta=\tup{\bbP_\beta,\frakG_\beta,\frakF_\beta}\) as follows.

\textit{Part}\enspace{}\textup{1:} \(\bbP_\beta\).\quad{}The forcing \(\bbP_\beta\) is the finite-support iteration of \(\tup{\ddbbQ_\gamma\mid\gamma<\beta}\).

\textit{Part}\enspace{}\textup{2:} \(\frakG_\beta\).\quad{}The automorphism group \(\frakG_\beta\) consists of elements of the form \(\mint_{\vec\pi}\) such that:
\begin{enumerate}
\item \(\vec\pi=\tup{\dot\pi_\gamma\mid\gamma<\beta}\) with each \(\dot\pi_\gamma\) appearing in \(\dsG_\gamma\); and
\item for all but finitely many \(\gamma<\beta\), \(\dot\pi_\gamma\) is forced to be the identity function.
\end{enumerate}
For such a sequence \(\vec\pi\), set \(C(\vec\pi)=\{\gamma<\beta\mid\1_{\bbP_\gamma}\not\forces_{\sS_\gamma}\dot\pi_\gamma=\id^\bullet\}\). We define the action of \(\mint_{\vec\pi}\) on \(\vec{p}\in\bbP_\beta\) recursively by, if \(\alpha=\max C(\vec\pi)\), then \(\mint_{\vec\pi} \vec{p} = \mint_{\vec\pi\res\alpha} \mint_{\dot\pi_\alpha} \vec{p}\), where
\begin{equation*}
\mint_{\dot\pi_\alpha}\vec{p}=\vec{p}\res\alpha\concat\dot\pi_\alpha(p_\alpha)\concat\dot\pi_\alpha(\vec{p}\res(\alpha,\beta)).
\end{equation*}
In fact, this is not how \(\frakG_\beta\) is originally defined, but by \cite[Proposition~3.8]{Karagila2018} these are equivalent.
\begin{fact}
\label{fact:BM-int}
Let \(\dot\pi_\gamma\in \dsG_{\gamma}\) for each \(\gamma\), and let the length of \(\vec\pi\) be greater than \(\beta\). Then:
\begin{enumerate}[label=\textup{(\roman*)}]
\item\label{fact:BM-int;cond:i} \(\mint_{\pi_\beta}\mint_{\vec\pi\res\beta}=\mint_{\vec\pi\res\beta\concat\tup{\pi_\beta}}\), and
\item\label{fact:BM-int;cond:ii} \(\mint_{\vec\pi}\dot\rho_{\beta+1}=\mint_{\vec\pi\res(\beta+1)}\dot\rho_{\beta+1}\).
\end{enumerate}
\end{fact}
\begin{proof}
For \ref*{fact:BM-int;cond:i} see the proof of \cite[Proposition~4.5]{Karagila2018}. For \ref*{fact:BM-int;cond:ii} one needs to consider the definition of \(\dot{\rho}_{\beta+1}\) and the permutation groups \(\dsG_\alpha\): \(\dot{\rho}_{\beta+1}\) is a \(\dot{\bbQ}_{\beta+1}\)\nobreakdash-name and if \(\dot\pi_\alpha\in \dsG_\alpha\) for \(\alpha>\beta\) then \(\dot\pi_\alpha\) fixes \(\dot{\rho}_{\beta+1}\). 
\end{proof}

The following fact is due to the proofs of \cite[Proposition~4.5, Proposition~4.14]{Karagila2018}.

\begin{fact}
\label{fact:homogeneity-p-times-q}
Suppose that we have constructed \(\dot\sI_\gamma\) for all \(\gamma<\beta\). Then for all \(\gamma<\beta\), \(\vec{p}\in\bbP_\gamma\), and \(\ddq,\ddq'\in\ddbbQ_\gamma\), there is some \(\mint_{\vec\pi}\in\frakG_\gamma\) such that \(\mint_{\vec\pi}\vec{p}=\vec{p}\) and \(\vec{p}\forces \mint_{\vec\pi}\ddq\comp \ddq'\). Moreover, if \(\gamma=\gamma'+1\) is a successor, then the automorphism can be further specified: such \(\mint_{\vec\pi}\) can be taken such that \(\vec\pi\res \gamma'=\id\) and \(\dot\pi_{\gamma'} p_{\gamma'}=p_{\gamma'}\). 
\end{fact}

\textit{Part}\enspace{}\textup{3:} \(\frakF_\beta\).\quad{}The filter \(\frakF_\beta\) consists of elements of the form \(\vec{H}=\tup{\ddH_\gamma\mid\gamma<\beta}\) such that:
\begin{enumerate}
\item \(\ddH_\gamma\) appears in \(\dsF_\gamma\) for all \(\gamma<\beta\); and
\item for all but finitely many \(\gamma<\beta\), \(\ddH_\gamma\) is forced to be \(\dsG_\gamma\).
\end{enumerate}
For \(\vec{H}\in\frakF_\beta\), let \(C(\vec{H})=\{\gamma<\beta\mid\1_{\bbP_\gamma}\not\forces_{\sS_\gamma}\ddH_\gamma=\dsG_\gamma\}\).

\subsubsection{Constructing \(\sI_\gamma\)}

We now proceed to define \(\dot\sI_\gamma=\tup{\ddbbQ_\gamma,\dsG_\gamma,\dsF_\gamma}\) for limit ordinals \(\gamma\). As they are less relevant to our purpose, we will not define \(\dsG_\gamma\) or \(\dsF_\gamma\) for successor \(\gamma\). However, we will still define \(\ddbbQ_\gamma\) in this case.

For all \(\beta\), \(\ddbbQ_{\beta+1}\) is as follows.
\begin{equation*}
\dot{\bbQ}_{\beta+1}=\left\{\mint_{\vec{\pi}}\dot{\rho}_{\beta+1}\res A\SetSymbol \mint_{\vec{\pi}}\in\mathfrak{G}_{\beta+1},\, A\subseteq\omega_{\beta+1}\text{ bounded, }A\in L\right\}.
\end{equation*}
Before we define \(\dot\sI_\gamma\) for limit \(\gamma\), we need to establish some notation. Firstly, given (partial) functions \(f,\,g\colon\gamma\to\Ord\), say that \(f\leq^\ast g\) if there is \(\alpha<\gamma\) such that for all \(\alpha<\beta<\gamma\), \(f(\beta)\leq g(\beta)\), and say that \(f=^\ast g\) if \(f\leq^\ast g\leq^\ast f\). For \(\beta\) a limit ordinal, let \(\prod\SC(\omega_\beta)\) be the product of successor cardinals below \(\omega_\beta\) indexed by successor ordinals. That is, an element of \(\prod\SC(\omega_\beta)\) is a function \(f\colon\{\alpha+1\mid\alpha<\beta\}\to\omega_\beta\) such that \(f(\gamma)<\omega_\gamma\) for all \(\gamma\). We denote by \(\overline{\prod\SC(\omega_\beta)}\) the \({=^\ast}\)-equivalence classes of \(\prod\SC(\omega_\beta)\). Fix a scale \(F=\{f_\eta\mid\eta<\omega_{\beta+1}\}\) in \(\overline{\prod\SC(\omega_\beta)}\), so \(F\) is a strictly increasing cofinal sequence in \(\overline{\prod\SC(\omega_\beta)}\) with respect to \(\leq^\ast\). Suppose that \(\vec\pi\) is a sequence \(\tup{\pi_\theta\mid \theta\in\SC(\omega_\beta)}\) such that \(\pi_\theta\) is an element of the symmetric group \(S_\theta\) for all successor cardinals \(\theta\) below \(\omega_\beta\). We say that \(\vec\pi\) \emph{implements} \(\pi\in S_{\omega_\beta}\) if for every large enough \(\theta\in\SC(\omega_\beta)\), \(f_{\pi(\eta)}(\theta)=\pi_\theta(f_\eta(\theta))\).

The permutation that \(\vec{\pi}\) implements is denoted \(\iota(\vec{\pi})\), if it exists. We say that \(F\) is a \emph{permutable scale} if every bounded permutation of \(\omega_{\beta+1}\) can be implemented by some \(\vec{\pi}\). That is, for all \(\eta<\omega_{\beta+1}\) every permutation of \(\eta\) can be implemented. By \cite[Theorem~3.27]{Karagila2018}, in \(L\) there is a permutable scale for every limit \(\beta\). We fix for each limit \(\beta\) a corresponding permutable scale.

Now assume that \(\alpha\) is a limit ordinal and, for each \(\beta<\alpha\), \(\dot\sI_\beta\) has been defined.
\begin{itemize}[itemsep=0.9ex]
\item For \(f\in \prod\SC(\omega_\alpha)\), let \(\dot\rho_{\alpha,f}\) be \(\tup{\dot\rho_{\beta+1}(f(\beta+1))\mid\beta<\alpha}^\bullet\).
\item Let \(\dot\rho_\alpha\) be \(\tup{\dot{\rho}_{\alpha,f}\mid f\in \prod\SC(\omega_\alpha)}^\bullet\). 
\item For \(E\subset \prod\SC(\omega_\alpha)\) and \(A\subset \alpha\), let \(\dot\rho_\alpha\res(E,A)\) be \(\tup{\dot{\rho}_{\alpha,f}\res A\mid f\in E}^\bullet\).
\item For \(f\in\prod\SC(\omega_\alpha)\), we define 
\begin{equation*}
\ddbbQ_{\alpha,f}=\left\{\mint_{\vec\pi}\dot\rho_{\alpha,f}\res A\SetSymbol A\subseteq\omega_\alpha\text{ bounded, }A\in L\right\}^\bullet
\end{equation*}
ordered by reverse inclusion.
\end{itemize}

We are now ready to define \(\dsI_\alpha=\tup{\ddbbQ_\alpha,\dsG_\alpha,\dsF_\alpha}\).

\textit{Part}\enspace{}\textup{1:} \(\ddbbQ_\alpha\).\quad{}We define \(\ddbbQ_\alpha\) to be the forcing
\begin{equation*}
\left\{\mint_{\vec\pi}\dot\rho_\alpha\res(E,A)\SetSymbol\mint_{\vec\pi}\in\frakG_\alpha,\,E\subseteq\prod\SC(\omega_\alpha)\text{ bounded, }A\subseteq\alpha\text{ bounded}\right\}^\bullet.
\end{equation*}
For \(\ddq=\tup{\mint_{\vec\pi}\dot\rho_{\beta+1}(g(\beta+1))\mid g\in E,\,\beta\in A}\in\ddbbQ_\alpha\) we denote by \(\ddq(f)\) the sequence \(\tup{\mint_{\vec\pi}\dot\rho_{\beta+1}(f(\beta+1))\mid\beta\in A}\). Say that \(\ddq\leq_{\ddbbQ_\alpha}\ddq'\) if and only if for all \(f\in\prod\SC(\omega_\alpha)\), \(\ddq(f)\leq_{\ddbbQ_{\alpha,f}}\ddq'(f)\).

Note that if \(\ddq=\tup{\mint_{\vec\pi}\dot\rho_{\beta+1}(g(\beta+1))\mid\beta\in A,\,g\in E}^\bullet\in\ddbbQ_\alpha\) then for each \(\beta\in A\) we have \(\tup{\mint_{\vec\pi}\dot\rho_{\beta+1}(g(\beta+1))\mid g\in E}^\bullet\in\ddbbQ_{\beta+1}\), since \(\mint_{\vec\pi}\dot\rho_{\beta+1}=\mint_{\vec\pi\res(\beta+1)}\dot\rho_{\beta+1}\) and \(E\) is bounded.

\textit{Part}\enspace{}\textup{2:} \(\dsG_\alpha\).\quad{}The group \(\dsG_\alpha\) is defined as the full-support product \(\prod_{\theta\in \SC(\omega_\alpha)}\sG_\theta\).

\textit{Part}\enspace{}\textup{3:} \(\dsF_\alpha\).\quad{}For \(\eta<\omega_{\alpha+1}\) and \(f\in \prod\SC(\omega_\alpha)\), let
\begin{equation*}
K_{\eta,f}\defeq\left\{\vec{\pi}\in\sG_\alpha\SetSymbol \iota(\vec\pi)\res\eta=\id\text{ and, for all }\mu\in \SC(\omega_\alpha),\,\pi_\mu\res f(\mu)=\id\right\}.
\end{equation*}
Then define \(\dsF_\alpha\) to be the filter generated by \(\{K_{\eta,f}\mid\eta<\omega_{\alpha+1},\,f\in\prod\SC(\omega_\alpha)\}\). This is indeed normal by \cite[Proposition~3.28]{Karagila2018}.

\subsubsection{Upwards homogeneity of \(\sS_\alpha\ast\dsI_\alpha\)}

Having constructed this limit stage, we may now show that it does indeed form an upwards homogeneous iteration of symmetric extensions. The structure of the proof of Proposition~\ref{prop:bristol-uphom} is from \cite[Lemma~4.10]{Karagila2018} but is adapted to suit our setting.

\begin{prop}\label{prop:bristol-uphom}
\(\tup{\bbP_\alpha,\frakG_\alpha,\frakF_\alpha}\ast\tup{\ddbbQ_\alpha,\dsG_\alpha,\dsF_\alpha}^\bullet\) is upwards homogeneous.
\end{prop}
\begin{proof}
Let \(\vec{H}\in\frakF_\alpha\) and \(K_{\eta,f}\in \sF_\alpha\). Given \(\tup{p_0,\ddq_0}\in\bbP_\alpha\ast\ddbbQ_\alpha\), taking \(\ddq_0\) to be \(\mint_{\vec{\pi}_0}\dot{\rho}_\alpha\res (E_0,A_0)\), let  \(\ddq^\circ\) be an extension of both \(\ddq_0\) and \(\mint_{\vec{\pi}_0}\dot{\rho}_\alpha\res (f{\downarrow},\delta+1)\), where \(\delta=\max C(\vec{H})\) and
\begin{equation*}
f{\downarrow}\defeq\left\{g\in \prod\SC(\omega_\alpha)\SetSymbol g(\mu)\leq f(\mu)\text{ for all }\mu\right\}.
\end{equation*}
We will show that this is a condition that meets our requirement in Definition~\ref{defn:upwards-homogeneous}.

Let \(\tup{\vec{p},\ddq},\,\tup{\vec{p},\ddq'}\leq\tup{p_0,\ddq^\circ}\). Suppose \(\ddq=\tup{\mint_{\vec{\pi}}\dot{\rho}_{\beta+1}(g(\beta+1))\mid \beta\in A,\, g\in E }^\bullet\). By Fact~\ref{fact:homogeneity-qgamma}, for all \(\beta<\alpha\) and \(r,r'\in \bbQ_{\beta+1}\) there is some \(\tau_\beta\in \sG_{\beta+1}\) such that \(\tau_{\beta}r\comp r'\). Therefore, since \(\tup{ \mint_{\vec{\pi}}\dot{\rho}_{\beta+1}(g(\beta+1))\mid g\in E}^\bullet \in \ddbbQ_{\beta+1}\) for each \(\beta\in A\cup(\delta+1)\), there is a sequence \(\vec{\tau}\) such that \(\tau_\gamma=\id\) for all \(\gamma\notin A\cap (\delta+1)\) and \(\vec{\tau}\ddq\res(\delta+1) \comp \ddq'\res(\delta+1)\).

Since \(\vec{\tau}\) is almost all identity, we get \(\iota(\vec{\tau})=\id\). Moreover, for \(\beta\in A\cap (\delta+1)\), \(\tau_\beta\) only needs to move co-ordinates above \(f(\beta+1)\), since \(\ddq\) and \(\ddq'\) both extend \(\mint_{\vec{\pi}_0}\dot{\rho}_\alpha\res (f{\downarrow},\delta+1)\). Therefore \(\vec{\tau}\in K_{\eta,f}\).

\begin{claim} For all \(\vec{p}\in \bbP_\alpha\) and \(\ddq,\ddq'\in \ddbbQ_\alpha\), if \(\ddq\res \delta+1 \comp \ddq'\res \delta+1\) then there is \(\mint_{\vec{\pi}}\in \frakG_\alpha\) such that \(\vec{\pi}\res \delta+1=\id\), \(\mint_{\vec{\pi}}\vec{p} =\vec{p}\), and \(\mint_{\vec{\pi}}\ddq\comp \ddq'\).
\end{claim}
\begin{proof}[Proof of Claim]\renewcommand{\qedsymbol}{\ensuremath{\dashv}}
Assume that \(\ddq=\mint_{\vec{\sigma}}\tup{ \dot{\rho}_{\beta+1}(g(\beta+1))\mid \beta\in A,\,g\in E}^\bullet\) and that \({\ddq'=\mint_{\vec{\sigma}'}\tup{ \dot{\rho}_{\beta+1}(g(\beta+1))\mid \beta\in A',\,g\in E'}^\bullet}\), with \(\ddq,\,\ddq'\) compatible below \(\delta+1\). Then we prove the claim by induction on \(\alpha\). There are two cases: 
\begin{enumerate}[label=\textup{\arabic*.}]
\item \(\alpha\) is a limit of limit ordinals and there is a least limit ordinal \(\theta\) such that \(A,\,A'\subset \theta\); or
\item \(\alpha=\theta+\omega\) for some limit \(\theta\).
\end{enumerate}
In both cases, we may assume that \(\delta<\theta\).

\textit{Case}\enspace{}1.\quad{}By the induction assumption on \(\theta\), for all \(\vec{p}\res \theta\in \bbP_\theta\) and \({\ddq\res \theta,\ddq'\res \theta\in \ddbbQ_\theta}\) there is \(\mint_{\vec{\pi}\res \theta}\in \frakG_\theta\) such that \(\vec{\pi}\res\delta+1=\id\), \(\mint_{\vec{\pi}\res\theta}\vec{p}\res \theta=\vec{p}\res\theta\), and \(\mint_{\vec{\pi}\res\theta}\ddq\res\theta\comp \ddq'\res\theta\). Since \(A,\,A'\subseteq\theta\) we may take the rest of the sequence \(\vec{\pi}\res[\theta,\alpha)\) to be the identity.

\textit{Case}\enspace{}2.\quad{}By the induction assumption on \(\theta\), we get part of the sequence \(\vec{\pi}\res \theta\). Since \(A\) and \(A'\) are bounded in \(\alpha\) we have only finitely many components left to define. Using Fact~\ref{fact:homogeneity-p-times-q} we find \(\pi_\theta\) such that \(\pi_\theta p_\theta=p_\theta\) and \(\mint_{\pi_\theta}\mint_{\vec{\pi}\res \theta}\ddq\comp\ddq'\) on the \((\theta+1)\)\textup{st} co-ordinate. That is,
\begin{equation*}
\mint_{\pi_\theta}\mint_{\vec{\pi}\res\theta}\mint_{\vec{\sigma}}\tup{ \dot{\rho}_{\theta+1}(g(\theta+1))\mid g\in E}^\bullet\ \mathbin{\mbox{\LARGE\(\comp\)}}\ \mint_{\vec{\sigma}'}\tup{\dot{\rho}_{\theta+1}(g(\theta+1))\mid g\in E'}^\bullet
\end{equation*} 
as conditions in \(\ddbbQ_{\theta+1}\). If \(\pi_{\theta+k}\) has been defined, we similarly find \(\pi_{\theta+k+1}\in \sG_{\theta+k+1}\) such that \(\pi_{\theta+k+1}p_{\theta+k+1}=p_{\theta+k+1}\) and \(\mint_{\pi_{\theta+k+1}}\mint_{\vec{\pi}\res \theta+k+1}\ddq\comp \ddq'\) on the \((\theta+k+1)\)\textup{st} co-ordinate.
\end{proof}
By the claim, we can find some \(\mint_{\vec{\pi}'}\in \vec{H}\) such that \(\mint_{\vec{\pi}'}\vec{\tau}\ddq\comp \ddq'\), and \(\mint_{\vec{\pi}'}\) fixes \(\vec{p}\).
\end{proof}

\section{Open questions}\label{sec:Future}

\begin{qn}
Consider an upwards homogeneous iteration \(\sS_0\ast\dsS_1\) defined in a model of \(\ZFC\). Then in the symmetric extension by \(\sS_0\) we have that \(\1_{\sS_1}\) forces that no new sets of ordinals are added. We now understand how this works given that the intermediate model is a symmetric extension, but what does this ``look like'' in the intermediate model? More precisely, what conditions are required for a symmetric system to add no new sets of ordinals? Or, less generally, what conditions are required for a symmetric system to add no new sets of reals?
\end{qn}

\begin{qn}
While not adding any subsets of the ground model is a powerful result, one sometimes can be satisfied by only partially completing this goal. For example, what conditions guarantee that in an iteration \(\sS_0\ast\dsS_1\) no new subsets of the ground model of rank less than \(\alpha\) are added? What conditions guarantee that a single symmetric system \(\sS\) (in a model of \(\ZF\)) adds no subsets of rank less than \(\alpha\)?
\end{qn}

\begin{qn}
How does the theory of upwards homogeneity work under iteration? We saw in Section~\ref{sec:Examples;subsec:monros-iteration} an example of an \(\omega\)-length iteration of symmetric extensions in which each iterand is upwards homogeneous over its predecessor (and hence any factorisation of a finite partial iteration is upwards homogeneous), but a technique for moving beyond the \(\omega\)\textup{th} step was developed only recently in \cite{Shani2018}. What can we say about the limit stage model of a finite support iteration of upwards homogeneous symmetric extensions?
\end{qn}

\section{Acknowledgements}\label{sec:Acknowledgements}

The authors would like to thank Asaf Karagila and Andrew Brooke-Taylor for helpful discussions regarding the paper. The authors would also like to thank the anonymous referee for their help in enhancing the clarity of the paper.

\end{document}